\documentclass{amsart}

\usepackage{fullpage}
\usepackage{amsmath}
\usepackage{amssymb}

\usepackage{graphicx}
\usepackage{enumitem}
\setlist[enumerate]{label={\upshape(\roman*)}}
\usepackage[colorlinks=true, linkcolor=black, citecolor=black]{hyperref}
\usepackage{cleveref}

\newcommand{\norm}[1]{\ensuremath{\left\lVert #1 \right\rVert }}
\newcommand{\abs}[1]{\ensuremath{\left\lvert #1 \right\rvert}}

\expandafter\mathchardef\expandafter\varphi\number\expandafter\phi\expandafter\relax
\expandafter\mathchardef\expandafter\phi\number\varphi

\newtheorem{theorem}{Theorem}

\newtheorem{lemma}[theorem]{Lemma}

\theoremstyle{definition}
\newtheorem{remark}[theorem]{Remark}
\newtheorem{definition}[theorem]{Definition}

\newtheorem{question}[theorem]{Question}

\DeclareMathOperator{\diag}{diag}
\DeclareMathOperator{\trace}{tr}

\newcommand{\Hil}{\ensuremath{\mathcal{H}}}

\renewcommand{\epsilon}{\varepsilon}

\begin{document}

\title{Traces on ideals and the commutator property}
\author{Jireh Loreaux}
\email{jloreau@siue.edu}
\address{Southern Illinois University Edwardsville \\
  Department of Mathematics and Statistics \\
  Edwardsville, IL, 62026-1653 \\
  USA}
\author{Gary Weiss}
\email{gary.weiss@uc.edu}
\address{University of Cincinnati   \\
  Department of Mathematics  \\
  Cincinnati, OH, 45221-0025 \\
  USA}

\begin{abstract}
  We propose a new class of traces motivated by a trace/trace class property discovered by Laurie, Nordgren, Radjavi and Rosenthal concerning products of operators outside the trace class.
  Spectral traces, traces that depend only on the spectrum and algebraic multiplicities, possess this property and we suspect others do, but we know of no other traces that do.

  This paper is intended to be part survey.
  We provide here a brief overview of some facts concerning traces on ideals, especially involving Lidskii formulas and spectral traces.

  We pose the central question: whenever the relevant products, $AB,BA$ lie in an ideal, do bounded operators $A,B$ always commute under any trace on that ideal, i.e., $\tau (AB) = \tau (BA)$? And if not, characterize which traces/ideals do possess this property.
\end{abstract}

\maketitle

Our motivation for this paper came from an evolution of trace results culminating with a complete, simple and elegant generalization to $B(H)$ by Laurie, Nordgren, Radjavi and Rosenthal of a trace result on the trace class.
We wondered what other traces/ideals possessed this property.
We show the well-studied spectral traces possess this property using Lidskii's theorem, and investigate whether or not there are others.
As a survey we provide proofs of some well-known facts we use.

Given a proper (two-sided) ideal $\mathcal{J}$ of the bounded operators $B(\Hil)$, a \emph{trace} on $\mathcal{J}$ is a unitarily invariant linear functional $\tau : \mathcal{J} \to \mathbb{C}$, i.e., $\tau(UAU^{*}) = \tau(AUAU^{*}) = \tau(A)$ for any unitary $U$ and $A \in \mathcal{J}$.
From the commutator relation $UAU^*-A= [UA,U^*]$, since any bounded operator is a linear combination of four unitary operators (see next paragraph), it follows easily that the unitary invariance of $\tau$ on $\mathcal{J}$ is equivalent to its vanishing on the commutator space $[\mathcal{J},B(\Hil)]$, the span of commutators where one of the operators in each commutator lies in $\mathcal{J}$.
Of course, the standard trace is an example of a trace on the trace class ideal (or any $B(H)$-ideal inside it), with its unitary invariance arising from the fact that the sum of the diagonal entries is independent of the choice of orthonormal basis.
Traces on ideals can have a plethora of other properties, including \emph{positivity} ($\tau : \mathcal{J}_+ \mapsto \mathbb{R}_+$) and \emph{normality} ($A_n \uparrow A$ in SOT implies $\tau(A_n) \to \tau(A)$), of which the standard trace on the trace class has both.

To prove that every operator is a linear combination of four unitaries, it suffices to show that every selfadjoint operator of norm less than or equal to one is the sum of a unitary operator and its adjoint.
This follows readily from the continuous function calculus.
Suppose $A = A^{*}$ with $\norm{A} \le 1$.
Then let $U := A + i\sqrt{I-A^2}$.
It is trivial to check that $U+U^{*} = A$ and $U^{*}U = UU^{*} = I$.

The threefold purpose of this paper is to pose the question: for which traces/ideals do all $B(\Hil)$ operators commute under only the condition that their two products $AB,BA$ both lie in the ideal; to provide evidence in support of it; and to point the reader where to look for possible counterexamples.

\begin{question}
  \label{que:trace-question}
  Given a trace $\tau : \mathcal{J} \to \mathbb{C}$ and $A,B \in B(\Hil)$ does
  \begin{equation*}
    \tau(AB) = \tau(BA) \quad\text{whenever}\ AB, BA \in \mathcal{J}?
  \end{equation*}
\end{question}

We now describe our motivation for this question.
Of course, the product of two Hilbert--Schmidt operators is trace-class (by the Cauchy--Schwarz inequality), and it is a well-known fact that whenever $A,B$ are Hilbert--Schmidt, then $\trace(AB) = \trace(BA)$; a direct application of Fubini's theorem.
Perhaps less well-known is that this extends to the case when $A \in \mathcal{K}$ and $B \in B(\Hil)$ and both products are trace-class.
This appears, for instance, in Gohberg--Krein \cite[Theorem~III.8.2]{GK-1969-ITTTOLNO} using the Schmidt decomposition for compact operators, or in Dunford--Schwartz \cite[Lemma~XI.9.14(b)]{DS-1963}.
However, we can obtain an affirmative answer to the stronger \Cref{que:trace-question} for the standard trace using a theorem of Lidskii \cite{Lid-1965-TS2AMS} (see \Cref{thm:lidskii}) and \autoref{lem:weyl-eigenvalue-lemma}.
Then, because Lidskii's theorem requires advanced techniques, we give the Laurie--Nordgren--Radjavi--Rosenthal \cite[Lemma~2.1]{LNRR-1981-JRAM} much simpler but powerful proof for the standard trace on the trace-class in the hope it contains ideas that are amenable to some level of generalization like \Cref{que:trace-question}.

\subsection*{Spectral traces and Lidskii's theorem}

For a compact operator $A$ let $\lambda(A)$ denote its eigenvalue sequence listed in order of nonincreasing modulus and repeated according to algebraic multiplicity.
The algebraic multiplicity of an eigenvalue $\lambda$ of $A$ is the dimension of the generalized eigenspace, which is defined as all vectors $v$ for which there is some $n \in\mathbb{N}$ so that $(\lambda-A)^n v = 0$.
Note, if $A$ has infinitely many nonzero eigenvalues, zero is omitted from the sequence $\lambda(A)$ since otherwise the important monotonizability of the moduli is impossible;
and in general, the sequence order may be nonunique because there is no dictated order for those eigenvalues with the same modulus.

A standard fact due to Weyl \cite{Wey-1949-PNAS} guarantees that $\lambda(A)$ is absolutely summable whenever $A$ is trace-class.
Moreover,

\begin{theorem}[Lidskii]
  \label{thm:lidskii}
  If $A \in \mathcal{K}$ is trace-class, then
  \begin{equation*}
    \trace A = \sum_{n=1}^{\infty} \lambda_n(A).
  \end{equation*}
\end{theorem}

It is well-known that the spectra of $AB$ and $BA$ have the same nonzero elements.
Indeed, if $\lambda \not= 0$ is not in the spectrum of $AB$, then
\begin{equation*}
  (\lambda - BA)^{-1} = \lambda^{-1} (I + B(\lambda-AB)^{-1}A),
\end{equation*}
which is verifiable by direct multiplication.
This next lemma, which is known, extends this fact for compact operators by ensuring that not only are the nonzero elements of the spectra $AB$ and $BA$ identical, they also have the same algebraic multiplicity.
We produce a short proof here for completeness.

\begin{lemma}
  \label{lem:algebraic-multiplicity}
  If $A,B \in B(\Hil)$ with $AB, BA \in \mathcal{K}$, then $\lambda_n(AB) = \lambda_n(BA).$
\end{lemma}

\begin{proof}
  We just showed that the nonzero elements of the spectra of $AB,BA$ are the same.
  It remains to prove that they have the same algebraic multiplicity.
  Let $v$ be a generalized eigenvector of $AB$ corresponding to the nonzero eigenvalue $\lambda$.
  That is, there is some $n \in \mathbb{N}$ for which $(\lambda-AB)^n v = 0$.
  Then certainly $Bv$ is nonzero, for otherwise by the binomial theorem $(\lambda-AB)^n v = \lambda v$ which is nonzero.
  This means $B$ is injective on the subspace of generalized eigenvectors of $AB$ for the eigenvalue $\lambda$.

  Observe that $Bv$ is a generalized eigenvector of $BA$ corresponding to $\lambda$.
  For this, notice that $(I-BA)B = B(I-AB)$ and hence by induction $(I-BA)^n B = B(I-AB)^n$.
  Therefore, $(I-BA)^n Bv = B(I-AB)^n v = 0$.
  Hence $B$ is an injective linear map from the finite dimensional subspace (see below) of generalized eigenvectors of $AB$ for $\lambda$ to the subspace of generalized eigenvectors of $BA$ for $\lambda$.
  Since the algebraic multiplicity of $\lambda$ is just the dimension of the generalized eigenspace, we can conclude that the algebraic multiplicity of $\lambda$ for $BA$ is at least as large as that for $AB$.
  Swapping the roles of $A,B$, we find that the algebraic multiplicity of $\lambda$ must be the same for both $AB$ and $BA$.

  Here is a short proof that generalized eigenspaces of a compact operator $K$ are finite dimensional.
  Let $\lambda$ be an eigenvalue.
  Consider the increasing sequence of subspaces $E_n = \{ v \in \Hil | (K-\lambda)^n v = 0 \}$ of generalized eigenvectors of order $n$.
  Notice that $(K-\lambda) E_{n+1} \subset E_n$.
  In particular, $(K-\lambda)$ is block upper triangular (with diagonal blocks zero) relative to the decomposition, $E_1$, $E_2 \ominus E_1$, $E_3 \ominus E_2, \ldots$, and therefore, in any basis for which $K-\lambda$ is in this block form, $K$ has diagonal lambda on the union of the $E_n$.
  Therefore, if $K$ is compact, then the union of the $E_n$ must be finite dimensional.
\end{proof}

Now, using \hyperref[thm:lidskii]{Lidskii's theorem} (\Cref{thm:lidskii}) and \Cref{lem:algebraic-multiplicity}, we see immediately that for the standard trace on the trace-class (or on any ideal inside it) \Cref{que:trace-question} has an affirmative answer.
This actually suggests a general approach to \Cref{que:trace-question}: whenever a trace $\tau$ on an ideal $\mathcal{J}$ depends only on the spectral data of the operator, then \Cref{que:trace-question} holds for $(\mathcal{J},\tau)$.
Such traces are called \emph{spectral} traces, and we will discuss them further shortly.
But first, we would like to present the entirely different proof mentioned earlier that the standard trace on the trace-class satisfies \Cref{que:trace-question}.
It is the primary motivation for this paper.

\subsection*{For trace-class, proof of \Cref{que:trace-question} without Lidskii's theorem}
This proof in \cite[Lemma~2.1]{LNRR-1981-JRAM} relies on four essential ingredients:
\begin{enumerate}
\item \label{item:A-positive} To prove \Cref{que:trace-question} it suffices to consider $A$ positive by considering the polar decomposition of $A$.

  This is because for the polar decomposition $A = U|A|$, one has
  \begin{equation*}
    \trace (AB - BA) = \trace (U|A|B - BU|A|) = \trace (|A|BU - BU|A|).
  \end{equation*}
  Needed here is that $\trace U|A|B = \trace |A|BU$ which holds since $AB$ is trace class and hence so also $U^*AB = |A|B$, the latter of which commutes with $U$ under the trace.
\item \label{item:P_n-commutes-A} If $P_n$ denotes the spectral projection of positive $A$ onto $[\frac{1}{n},\norm{A}]$, $P_n$ commutes with $A$.
\item \label{item:P_n-A-local-invert} $P_n A$ is locally invertible, so if $P_n A P_n B P_n$ is trace-class, then $P_n B P_n$ must be trace-class, so these operators commute under the trace.
\item \label{item:P_n-to-I-sot} $P_n \to I$ in the strong operator topology and the standard trace is strong operator continuous.
\end{enumerate}
Together, this means
\begin{equation*}
  \trace(P_n AB) \underset{\text{\ref{item:P_n-commutes-A}}}= \trace(P_n A P_n B P_n) \underset{\text{\ref{item:P_n-A-local-invert}}}= \trace(P_n B P_n A P_n) \underset{\text{\ref{item:P_n-commutes-A}}}= \trace(BA P_n).
\end{equation*}
By \ref{item:P_n-to-I-sot}, the left-hand side converges to $\trace(AB)$ and the right-hand side converges to $\trace(BA)$, so we have $\trace(AB) = \trace(BA)$.
This proof is beautiful for its simplicity, elegance and that it does not rely on Lidskii's theorem.

If one is interested in generalizing the above proof to other operator ideals and traces $(\mathcal{J}, \tau)$, there is some hope.
Notice that \ref{item:A-positive}--\ref{item:P_n-A-local-invert} above hold with the trace-class ideal replaced by $\mathcal{J}$ and the standard trace replaced by $\tau$.
It is only \ref{item:P_n-to-I-sot} which would cause issues.
Indeed, nonzero singular traces (those which vanish on the finite rank operators) are never strong operator continuous.
However, if one could show that $P_n AB \to AB$ and $BAP_n \to BA$ in some topology on $\mathcal{J}$ with respect to which $\tau$ is continuous, then this technique would be salvaged to prove \Cref{que:trace-question} for $(\mathcal{J},\tau)$.

Now back to spectral traces.

\begin{definition}
  \label{def:spectral-trace}
  A trace $\tau$ on an ideal $\mathcal{J}$ is said to be \emph{spectral} if for all $A \in J$, $\tau(A) = \tau(\diag(\lambda(A))$.
\end{definition}

Although \Cref{def:spectral-trace} appears innocuous at first, it has a hidden bite.
In particular, spectral traces make an implicit assumption on the ideal $\mathcal{J}$, namely, if $A \in \mathcal{J}$, then $\diag(\lambda(A)) \in \mathcal{J}$.
Unfortunately, this is not true for an arbitrary ideal $\mathcal{J}$.

Weyl \cite{Wey-1949-PNAS} established the following relationship (\Cref{lem:weyl-eigenvalue-lemma}) between a compact operator's eigenvalue sequence and its singular value sequence, contained in the next lemma.
For nonnegative, nonincreasing sequences $(a_j), (b_j)$ we say $(b_j)$ \emph{logarithmically submajorizes} $(a_j)$ and write $(a_j) \prec_{log} (b_j)$ if for all $n \in \mathbb{N}$,
\begin{equation*}
  \prod_{j=1}^n a_j \le \prod_{j=1}^n b_j.
\end{equation*}
As Sukochev and Zanin remark in \cite{SZ-2014-AM}, Weyl used the idea of logarithmic submajorization long before the order was formally defined by Ando and Hiai in \cite{AH-1994-LAA}.

\begin{lemma}[Weyl]
  \label{lem:weyl-eigenvalue-lemma}
  If $A$ is a compact operator, then its singular value sequence logarithmically submajorizes the element-wise absolute value of the eigenvalue sequence.
  Symbolically, for all $n \in \mathbb{N}$,
  \begin{equation*}
    \prod_{j=1}^n \abs{\lambda_j(A)} \le \prod_{j=1}^n s_j(A).
  \end{equation*}
\end{lemma}

There is a converse to Weyl's lemma which states that for any pair of nonnegative sequences converging to zero one of which logarithmically submajorizes the other, there is a compact operator with those sequences as its eigenvalue sequence and singular value sequence.
In light of \Cref{lem:weyl-eigenvalue-lemma} and its converse, the next definition encapsulates precisely those ideals for which $A \in \mathcal{J}$ implies $\diag(\lambda(A)) \in \mathcal{J}$.

\begin{definition}
  An ideal $\mathcal{J}$ is said to be \emph{logarithmically closed} if whenever $B \in \mathcal{J}$ and $s(A)$ is logarithmically submajorized by $s(B)$, then $A \in \mathcal{J}$.
\end{definition}

Sukochev and Zanin \cite{SZ-2014-AM} recently completed a \emph{tour de force} on the subject of spectral traces with the following key result.

\begin{theorem}[Sukochev--Zanin]
  \label{thm:which-traces-spectral}
  Let $\mathcal{J}$ be an arbitrary ideal with trace $\tau$.
  \begin{itemize}
  \item If $\mathcal{J}$ is logarithmically closed, every trace on $\mathcal{J}$ is spectral.
  \item If $\tau$ is a positive trace on $\mathcal{J}$ that respects logarithmic submajorization (i.e., if $s(A) \prec_{log} s(B)$, then $\tau(A) \le \tau(B)$), then $\tau$ extends to a spectral trace on the logarithmic envelope of $\mathcal{J}$ (the smallest logarithmically closed ideal (intersection of all) containing $\mathcal{J}$).
  \end{itemize}
\end{theorem}

\Cref{thm:which-traces-spectral} in conjunction with \Cref{lem:algebraic-multiplicity} guarantees that any trace on a logarithmically closed ideal or any positive trace that respects logarithmic submajorization necessarily satisfies the trace-commutator property.
Sukochev and Zanin also provided an example of a trace which is \emph{not} spectral.
Although spectral traces automatically satisfy \Cref{que:trace-question} by Lidskii's theorem and \Cref{lem:algebraic-multiplicity}, it is not obvious to us whether or not any non-spectral traces fail to satisfy \Cref{que:trace-question}.
But we describe in \Cref{rem:places-to-look,rem:nonzero-traces} some places to look for counterexamples to \Cref{que:trace-question}.

Sukochev and Zanin refer to Dykema and Kalton \cite[Example~1.5]{DK-1998-JRAM} for an example of a non-spectral trace.
The example provided by Dykema and Kalton is a quasinilpotent operator $Q$ in an ideal $\mathcal{J}$ with trace $\tau$ for which $\tau(Q) \not= 0$.
If one could write this $Q$ as a commutator of bounded operators with at least one product (and hence both) as a commutator of bounded operators, this would provide a negative answer to \Cref{que:trace-question}.
However, this line of attack seems to us to be difficult because there are few results of this nature in the literature for positive operators, let alone quasinilpotent ones.

\begin{remark}
  \label{rem:places-to-look}
  Suppose $\tau$ is a positive trace on an ideal $\mathcal{J}$ which does not respect logarithmic submajorization.
  If there exists $A \in \mathcal{J}_{+}$ with $\tau(A) > 0$ which is a commutator of bounded operators where at least one product (and hence both) lie in $\mathcal{J}$, then $\tau$ does not satisfy \Cref{que:trace-question}.
  It was Brown, Halmos and Pearcy in \cite{BHP-1965-CJM} who originally proved that every positive compact operator is a commutator of bounded operators.
  However, that construction does not produce products that are compact.

  It is possible to construct certain compact operators with zero kernel which are commutators of compact operators (see \Cref{thm:bpw-commutators} below).
  This was recently achieved by Patnaik \cite{Pat-2012} and Beltita--Patnaik--Weiss \cite{BPW-2014-VLOT}, but also prior by Davidson, Marcoux and Radjavi in unpublished form, thereby solving a 40-year old problem.
  Moreover, information regarding the ideals can be extracted.
  That is, when the target positive compact operator lies in some ideal, one can specify to some degree where the products in the commutator lie.

  However, it seems unnecessarily restrictive to require that both operators in the commutator be compact.
  Perhaps there is some construction which somehow interpolates between these two so that the individual operators are not compact, but the products lie in a desired ideal with a trace.
\end{remark}

Below we present the construction of certain compact operators, including some with zero kernel, as commutators of other compact operators given in \cite{BPW-2014-VLOT}.
The hope is that this construction can either be used explicitly for the purposes of the previous remark, or as inspiration for a commutator where one of the operators is bounded and the other compact.

In the theorem below, the condition that $d_n \ge 0$ may be dropped entirely by using the principal branch of the square root function throughout the proof.

\begin{theorem}[\cite{BPW-2014-VLOT}]
  \label{thm:bpw-commutators}
  If $d_n \ge 0$ and $\frac{1}{n}\displaystyle{\sum_{j=1}^n}d_j \rightarrow 0$, then the positive compact operators with eigenvalue sequence
  \begin{equation*}
    \Big(d_1, \frac{d_2}{2}, \frac{d_2}{2}, \frac{d_3}{3}, \frac{d_3}{3}, \frac{d_3}{3}, \cdots\Big)
  \end{equation*}
  are single commutators of compact operators.
\end{theorem}

\begin{proof}[Construction]
  Consider the block tri-diagonal matrices
  \begin{equation*}
    C =
    \begin{pmatrix}
      0   & A_1                   \\
      B_1 & 0   & A_2             \\
      {}  & B_2 & 0      & \ddots \\
      {}  & {}  & \ddots & \ddots \\
    \end{pmatrix}
    \quad\text{and}\quad
    Z =
    \begin{pmatrix}
      0     & X_1                   \\
      Y_{1} & 0   & X_2             \\
      {}    & Y_2 & 0      & \ddots \\
      {}    & {}  & \ddots & \ddots \\
    \end{pmatrix}
  \end{equation*}
  where $A_n$ and $X_{n}$ are the $n\times(n+1)$ matrices of norm $\sqrt{\frac{\sum_{k=1}^n d_k}{n}}$
  \begin{equation*}
    A_n =
    \frac{\sqrt{\sum_{k=1}^n d_k}}{n}
    \begin{pmatrix}
      \sqrt n & 0                                  \\
      {}      & \sqrt {n-1} & 0                    \\
      {}      & {}          & \ddots & \ddots      \\
      {}      & {}          & {}     & \sqrt 1 & 0 \\
    \end{pmatrix}
    \quad\text{and}\quad
    X_n =
    \frac{\sqrt{\sum_{k=1}^n d_k}}{n}
    \begin{pmatrix}
      0  & \sqrt 1                        \\
      {} & 0  & \sqrt 2                   \\
      {} & {} & \ddots & \ddots & {}      \\
      {} & {} & {}     & 0      & \sqrt n \\
    \end{pmatrix}
  \end{equation*}
  while $B_n$ and $Y_n$ are the $(n+1)\times n$ matrices  of norm $\frac{\sqrt{n\sum_{k=1}^n d_k}}{n+1}$
  \begin{equation*}
    B_n =
    -\frac{\sqrt{\sum_{k=1}^n d_k}}{n+1}
    \begin{pmatrix}
      0                                        \\
      \sqrt 1 & 0           &                  \\
      {}      & \sqrt 2     & \ddots & {}      \\
      {}      & {}          & \ddots & 0       \\
      {}      & {}          & {}     & \sqrt n \\
    \end{pmatrix}\,
    \quad\text{and}\quad
    Y_n =
    \frac{\sqrt{\sum_{k=1}^n d_k}}{n+1}
    \begin{pmatrix}
      \sqrt n & {}          & {}     & {}      \\
      0       & \sqrt {n-1} & {}     & {}      \\
      {}      & 0           & \ddots & {}      \\
      {}      & {}          & \ddots & \sqrt 1 \\
      {}      & {}          & {}     & 0       \\
    \end{pmatrix}.
  \end{equation*}
  Then
  \begin{equation*}
    [C, Z] =
    \begin{pmatrix}
      D_1    & 0      & U_1    & 0      & \cdots                                  \\
      0      & D_2    & 0      & U_2    & 0       & \cdots                        \\
      L_1    & 0      & D_3    & 0      & U_3     & 0 \cdots                      \\
      \vdots & \ddots & \ddots & \ddots & \ddots  & \ddots & \ddots               \\
      0      & \cdots & L_n    & 0      & D_{n+1} & 0      & U_{n+1} & 0 & \cdots \\
      \vdots                                                                      \\
    \end{pmatrix}
  \end{equation*}
  where $L$'s, $D$'s  and $U$'s are
  \begin{equation*}
    \begin{matrix}
      L                     & D                                                   & U                       \\
      B_2Y_1-Y_2B_1         & A_1 Y_1-X_1 B_1 = d_1                                 & A_1 X_2-X_1 A_2         \\
      B_3Y_2-Y_3B_2         & B_1 X_1-Y_1 A_1 + A_2 Y_2-X_2 B_2                   & A_2 X_3-X_2 A_3         \\
      \vdots                & \vdots                                              & \vdots           \\
      B_{n+1}Y_n-Y_{n+1}B_n & B_n X_n-Y_n A_n + A_{n+1}Y_{n+1}-X_{n+1}B_{n+1}     & A_n X_{n+1}-X_n A_{n+1} \\
      0 & -\left( \frac{\sum_1^n d_k}{n} \right) I_{n+1}  + \left( \frac{\sum_1^{n+1} d_k}{n} \right) I_{n+1} & 0\\
    \end{matrix}
  \end{equation*}
  Straightforward (albeit tedious) calculations ensure the stated relationships.
\end{proof}

The significance of the above theorem is manifest: if one can find a trace $\tau$ on the ideal generated by $CZ, ZC$ for which $\tau(D) \not= 0$, then that would be a counterexample to \Cref{que:trace-question}.
Examination of the products $CZ, ZC$ guarantees that they are contained in the ideal generated by the positive operator with eigenvalue sequence
\begin{equation*}
  \Bigg(d_1, \frac{\sum_1^2 d_k}{2}, \frac{\sum_1^2 d_k}{2}, \frac{\sum_1^3 d_k}{3}, \frac{\sum_1^3 d_k}{3}, \frac{\sum_1^3 d_k}{3}, \cdots\Bigg).
\end{equation*}
Indeed, note that the products $CZ, ZC$ are block tri-diagonal.
Moreover, the matrices in each block are either themselves diagonal, or they are rectangular weighted shifts.
In addition, the entries in the $n$th block are dominated by $\frac{d_n}{n}$ and the $n$th block has $n$ nonzero entries.
Thus, the singular values of each block are dominated by $\frac{d_n}{n}$ repeated $n$ times.

The downside is that this is in general a strictly larger ideal than the ideal generated by $D$.
Moreover, by a result of Varga \cite{Var-1989-PAMS}, a positive trace exists on a principal ideal if and only if the sequence of singular values of the generating operator is irregular.
In light of \Cref{thm:which-traces-spectral}, another important question with what seems a nonobvious answer is whether the ideal generated by $CZ,ZC$ is logarithmically closed.

\begin{remark}
  \label{rem:nonzero-traces}
  Not all ideals support nonzero traces, the most notable example being the ideal $\mathcal{K}$ of compact operators.
  Therefore, it is also important to recognize when an ideal supports a trace.
  An ideal supports a nonzero trace if and only if the quotient $\mathcal{J} \slash [\mathcal{J},B(\Hil)] \not= 0$ if and only if $\mathcal{J} \not= \mathcal{J}_a$ \cite{DFWW-2004-AM} (see also \cite{KW-2010-JOT} and also \Cref{rem:diagonal-invariance}).
\end{remark}

\begin{remark}
  \label{rem:diagonal-invariance}
  The similar study of which ideals $\mathcal{J}$ possess \emph{diagonal invariance} (i.e., $A \in \mathcal{J}$ implies the diagonal of $A$ is also in $\mathcal{J}$ in any orthonormal basis) was solved in \cite{KW-2011-IUMJ}.
  The condition for an ideal $\mathcal{J}$ to possess diagonal invariance is that it be \emph{arithmetically mean closed} (i.e., if $\diag \eta \in \mathcal{J}$ and $\xi \prec \eta$, then $\diag \xi \in \mathcal{J}$).
  Observe $\xi \prec \eta$ if and only if their C\'esaro averages / arithmetic means $\xi_a \le \eta_a$.
  And $\mathcal{J}_a$ is the ideal generated by all $\xi_a$ for all $\xi \in c_0^{*}$ with $\diag \xi \in \mathcal{J}$.

  The converse of diagonal invariance holds in general.
  That is, for any ideal $\mathcal{J}$ and for any operator $A$, if every diagonal of $A$ lies in $\mathcal{J}$, then $A \in \mathcal{J}$.
  We provide a short proof here.
  First it is a well-known folklore result that if every diagonal of an operator converges to zero, then the operator is compact.
  So, given an operator $A$, if every diagonal of $A$ lies in $\mathcal{J}$, then $A$ is compact.
  Thus we may diagonalize the real part of $A$ and call $\lambda_n$ the sequence of eigenvalues of the real part and $d_n$ the diagonal of $A$ in the basis of eigenvectors (for the real part).
  Since the singular values of the real part of $A$ are the absolute values of $\lambda_n$, and these are dominated by the absolute value of the diagonal $d_n$, which lies in $\mathcal{J}$, then the real part of $A$ lies in $\mathcal{J}$ as well.
  A symmetric argument diagonalizing the imaginary part proves that it lies in $\mathcal{J}$, and hence so does $A$.
\end{remark}

\bibliographystyle{amsalpha}
\bibliography{references}

\end{document}